\documentclass[12pt,twoside]{article}

\newcommand{\HeadTitle}{Asymptotic zero distribution of the polynomials $\widetilde{\Xi}_n$}

\newcommand{\HeadTitleTwo}{\begin{center}
\Large{\textit{Asymptotic zero distribution of the polynomials $\widetilde{\Xi}_n$}}
\end{center}}

\usepackage[margin=0.8in]{geometry}
\usepackage{amsmath, amssymb, amsfonts, mathtools}
\usepackage{amsthm}
\usepackage{mathrsfs}
\usepackage{stmaryrd}
\usepackage{esint}
\usepackage{eqnalign}

\usepackage{graphicx}
\usepackage{enumitem}
\usepackage[most]{tcolorbox}
\usepackage{float}
\usepackage{listings}
\usepackage{array}
\usepackage{booktabs}
\usepackage{xcolor}
\usepackage{emptypage}
\usepackage{tensor}
\usepackage{tabularx}

\usepackage{fancyhdr}
\usepackage{hyperref}
\usepackage[nameinlink,noabbrev]{cleveref}

\usepackage{titlesec}
\usepackage[T1]{fontenc}
\usepackage[utf8]{inputenc}
\usepackage{lmodern}
\usepackage{titling}

\usepackage[backend=biber,style=numeric]{biblatex}
\newcommand{\EulerB}[2]{\left\langle {#1 \atop #2} \right\rangle^{\!B}}

\providecommand{\HeadTitle}{} % <-- TITRE VARIABLE
\providecommand{\HeadTitleTwo}{} % <-- TITRE VARIABLE
\providecommand{\HeadAuthor}{Luc Ramsès TALLA WAFFO} % <-- FIXE

\titleformat{\section}[block]
  {\normalfont\large\bfseries\itshape\centering}
  {§\thesection.}
  {1em}
  {}

\titleformat{\subsection}
  {\normalfont\normalsize\bfseries}
  {\thesubsection}
  {1em}
  {}

\newtheorem{theorem}{Theorem}[section]
\newtheorem{lemma}[theorem]{Lemma}

\crefname{corollary}{corollary}{corollaries}
\Crefname{corollary}{Corollary}{Corollaries}

\crefname{conjecture}{conjecture}{conjectures}
\Crefname{conjecture}{Conjecture}{Conjectures}

\crefname{lemma}{lemma}{lemmas}
\Crefname{lemma}{Lemma}{Lemmas}

\crefname{proposition}{proposition}{propositions}
\Crefname{proposition}{Proposition}{Propositions}

\crefname{theorem}{theorem}{theorems}
\Crefname{theorem}{Theorem}{Theorems}

\begin{document}

\thispagestyle{fancy}

\vspace{0.2cm}

\begin{center}
\Large{\HeadTitleTwo}
\end{center}

\hspace{3cm}

\begin{center}
Luc Ramsès TALLA WAFFO \\
Technische Universität Darmstadt\\
Karolinenplatz 5, 64289 Darmstadt, Germany\\
ramses.talla@stud.tu-darmstadt.de\\
\vspace{0.5cm}
February 21, 2026
\end{center}

\begin{abstract}
We consider the polynomials $\Xi_n$ introduced in~\cite{TallaWaffo2025arxiv2511.02843} and studied in further details in\cite{TallaWaffo2026arxiv2602.16761}, which are expressed in terms of Eulerian polynomials of type~B, and study the zero distribution of the rescaled family
\[
  \widetilde{\Xi}_n(x) := \Xi_n(\sqrt{x}), \qquad n\ge 2.
\]
Writing the zeros of $\widetilde{\Xi}_n$ in the interval $(0,1)$ as
$0< x_{n,1} \le \cdots \le x_{n,n-1} < 1$ and forming the empirical measures
\[
  \mu_n := \frac1{n-1}\sum_{k=1}^{n-1}\delta_{x_{n,k}},
\]
we prove that $(\mu_n)_{n\ge2}$ converges weakly to a deterministic probability measure $\mu$ supported on $(0,1)$.
We give an explicit formula for the limiting density and the limiting distribution function of~$\mu$.
The proof is based on a representation of $\Xi_n$ in terms of type~B Eulerian polynomials, 
a ratio asymptotic for these polynomials derived from a classical series identity, and the
Stieltjes transform method.
We also provide numerical experiments illustrating the convergence of the empirical zero distributions to~$\mu$.
\end{abstract}

\vspace{0.2cm}

\paragraph{Notation.}
Throughout this manuscript we use the following conventions. 
We write $\mathbb{N}_0$ for the set of all non-negative integers and $\mathbb{N}$ for the set of positive integers. 
The symbol $\EulerB{n}{k}$ denotes the Eulerian numbers of type~$B$, and $B_n(x)$ denotes the Eulerian polynomial of type~$B$. 

\vspace{0.5cm}

\section*{Introduction}
In \cite{TallaWaffo2026arxiv2602.16761} the polynomials
\begin{equation}\label{eq:Xi-def}
  \Xi_n(x)
  = \frac{(-1)^{n+1}}{2^{4n-1}(2n-1)!}
    \frac{(1+x)^{2n-1}}{x}\,
    B_{2n-1}\!\left(-\frac{1-x}{1+x}\right),
  \qquad n\ge 1,
\end{equation}
were introduced and studied. Here $B_m(z)$ denotes the Eulerian polynomial of type~B of degree $m$.
Following \cite{TallaWaffo2026arxiv2602.16761}, we define
\begin{equation}\label{eq:Xitilde-def}
  \widetilde{\Xi}_n(x) := \Xi_n(\sqrt{x}), \qquad n\ge 2.
\end{equation}
It is easy to check that $\widetilde{\Xi}_n$ is a real polynomial of degree $n-1$\cite{TallaWaffo2026arxiv2602.16761}.

We are interested in the asymptotic distribution of the real zeros of $\widetilde{\Xi}_n$ in the interval $(0,1)$.
For each $n\ge2$ we denote these zeros by
\[
  0 < x_{n,1} \le x_{n,2} \le \cdots \le x_{n,n-1} < 1
\]
and define the empirical zero counting measure
\begin{equation}\label{eq:mu-n-def}
  \mu_n := \frac1{n-1}\sum_{k=1}^{n-1} \delta_{x_{n,k}}.
\end{equation}
Equivalently, the corresponding empirical distribution function is
\[
  F_n(x) := \mu_n((-\infty,x]) 
   = \frac1{n-1}\#\{k: x_{n,k}\le x\}, \qquad x\in\mathbb{R}.
\]

Our main results are \cref{lem:s-n-limit}, \cref{thm:limiting-density} and \cref{thm:limiting-zero-distribution}. In particular, the empirical distribution functions $F_n$ associated with the real zeros of $\widetilde{\Xi}_n$ in $(0,1)$ converge to $F$ at every continuity point of~$F$.

\vspace{0.2cm}

We note in passing that the limit measure $\mu$ is absolutely continuous and its density has a \emph{log--Cauchy} structure, reminiscent of the limiting zero distributions that appear for classical Eulerian polynomials.
The plots in \cref{sec:numerics} show that the convergence of the empirical distribution functions $F_n$ to $F$ is already very pronounced for moderate values of $n$. 

\vspace{0.2cm}

In \cite{TallaWaffo2026arxiv2602.16761}, there is another family of polynomials $\Lambda_n(x)$ exhibiting the same qualitative properties as those considered in this article. However, we did not pursue the study of this family here, since Paul Melotti \cite{MelottiEulerianRoots} has already investigated in detail the distribution of roots of Eulerian polynomials of type~A. 
In fact, the root distribution of the family $\Lambda_n$ can be readily deduced from his results by means of suitable bijective transformations. From this perspective, the present work may be viewed as an extension of Melotti’s results from Eulerian polynomials of type~A to those of type~B.

\vspace{0.3cm}

\section{Preliminaries}\label[section]{sec:preliminaries}

\begin{lemma}\label[lemma]{lemma:Sm-ratio}
Let
\[
S_m(x) := \sum_{k\ge 0} (2k+1)^m x^k,
\]
for $m\in\mathbb{N}$ and $0<x<1$. Then
\[
\lim_{m\to\infty} \frac{1}{m}\,\frac{S_{m+1}(x)}{S_m(x)} \;=\; \frac{2}{-\log x}.
\]
\end{lemma}

\begin{proof}
We first express $S_m(x)$ as a contour integral. For $n\in\mathbb{N}$ and $m\in\mathbb{N}$, Cauchy’s integral formula for derivatives gives
\[
n^m
= \left.\frac{d^m}{dz^m} e^{nz}\right|_{z=0}
= \frac{m!}{2\pi i} \oint_{|z|=r} \frac{e^{nz}}{z^{m+1}}\,dz,
\]
where $\{|z|=r\}$ is the positively oriented circle centered at $0$.

Hence
\[
(2k+1)^m
= \frac{m!}{2\pi i} \oint_{|z|=r} \frac{e^{(2k+1)z}}{z^{m+1}}\,dz.
\]
Substituting this into $S_m(x)$ and interchanging sum and integral (justified by absolute convergence for $|x|<1$ and $r$ sufficiently small) yields
\[
S_m(x)
= \frac{m!}{2\pi i} \oint_{|z|=r} \frac{1}{z^{m+1}}
   \sum_{k\ge 0} x^k e^{(2k+1)z}\,dz.
\]
The sum is geometric:
\[
\sum_{k\ge 0} x^k e^{(2k+1)z}
= e^z \sum_{k\ge 0} (x e^{2z})^k
= \frac{e^z}{1 - x e^{2z}},
\]
for $|x e^{2z}|<1$, which holds on a sufficiently small circle $|z|=r$. Thus
\[
S_m(x)
= \frac{m!}{2\pi i} \oint_{|z|=r}
  \frac{e^z}{z^{m+1} \bigl(1-x e^{2z}\bigr)}\,dz.
\]
Set
\[
g(z) := \frac{e^z}{1 - x e^{2z}},
\quad\text{so that}\quad
S_m(x) = \frac{m!}{2\pi i} \oint_{|z|=r} \frac{g(z)}{z^{m+1}}\,dz.
\]

We now analyze the singularities of $g$. The poles are determined by
\[
1 - x e^{2z} = 0
\quad\Longleftrightarrow\quad
e^{2z} = \frac{1}{x}.
\]
Since $0<x<1$, we have $1/x>1$, and all solutions are
\[
z_k = \frac{1}{2}\log\frac{1}{x} + \pi i k
    =: z_0 + \pi i k,\quad k\in\mathbb{Z},
\]
where
\[
z_0 := \frac{1}{2}\log\frac{1}{x} = -\frac{1}{2}\log x > 0
\]
is the unique pole of smallest modulus (real and positive). Each $z_k$ is a simple pole of $g$.

The residue of $g$ at $z_0$ is
\[
\operatorname{Res}(g,z_0)
= \frac{e^{z_0}}{(1 - x e^{2z})'(z_0)}
= \frac{e^{z_0}}{-2x e^{2z_0}}
= \frac{e^{z_0}}{-2x\cdot (1/x)}
= -\frac{e^{z_0}}{2}.
\]

Choose radii $0<r<|z_0|$ and $R$ such that
\(\displaystyle
|z_0| < R < \min_{k\ne 0} |z_k|.
\)
Enlarge the contour from the positively oriented circle $|z|=r$ to the positively oriented circle $|z|=R$. By the residue theorem,
\[
\oint_{|z|=r}\frac{g(z)}{z^{m+1}}\,dz
= \oint_{|z|=R}\frac{g(z)}{z^{m+1}}\,dz
  - 2\pi i\,\operatorname{Res}\Bigl(\frac{g(z)}{z^{m+1}}, z_0\Bigr),
\]
because in deforming the contour we enclose the additional pole at $z_0$ and no others (by the choice of $R$).

Since $g$ has a simple pole at $z_0$, we have
\(\displaystyle
\operatorname{Res}\Bigl(\frac{g(z)}{z^{m+1}}, z_0\Bigr)
= \frac{\operatorname{Res}(g,z_0)}{z_0^{m+1}}
= -\frac{e^{z_0}}{2 z_0^{m+1}}.
\)

Thus
\[
S_m(x)
= \frac{m!}{2\pi i} \oint_{|z|=r}\frac{g(z)}{z^{m+1}}\,dz
= \frac{m!}{2\pi i} \oint_{|z|=R}\frac{g(z)}{z^{m+1}}\,dz
  + m!\,\frac{e^{z_0}}{2 z_0^{m+1}}.
\]

On the positively oriented circle $|z|=R$, the function $g$ is analytic and bounded, say $|g(z)|\le M$ for all $|z|=R$. Hence
\(\displaystyle
\left|\frac{m!}{2\pi i} \oint_{|z|=R}\frac{g(z)}{z^{m+1}}\,dz\right|
\le \frac{m!}{2\pi} \cdot 2\pi R \cdot \frac{M}{R^{m+1}}
= m!\,M\,R^{-m}.
\)

Therefore
\[
S_m(x)
= m!\,\frac{e^{z_0}}{2 z_0^{m+1}}
  + O\bigl(m!\,R^{-m}\bigr),
\quad m\to\infty.
\]
The same argument with $m$ replaced by $m+1$ gives
\[
S_{m+1}(x)
= (m+1)!\,\frac{e^{z_0}}{2 z_0^{m+2}}
  + O\bigl((m+1)!\,R^{-(m+1)}\bigr).
\]

Thus
\(\displaystyle
\frac{S_{m+1}(x)}{S_m(x)}
= \frac{(m+1)!\,\dfrac{e^{z_0}}{2 z_0^{m+2}}
          \bigl(1 + O\bigl((R/z_0)^{-m}\bigr)\bigr)}
       {m!\,\dfrac{e^{z_0}}{2 z_0^{m+1}}
          \bigl(1 + O\bigl((R/z_0)^{-m}\bigr)\bigr)}
= \frac{m+1}{z_0}\bigl(1 + o(1)\bigr),
\quad m\to\infty,
\)
since $R>|z_0|$ implies $(R/z_0)^{-m}\to 0$.

Dividing by $m$ and letting $m\to\infty$ yields
\[
\lim_{m\to\infty} \frac{1}{m}\,\frac{S_{m+1}(x)}{S_m(x)}
= \lim_{m\to\infty} \frac{m+1}{m}\,\frac{1}{z_0}
= \frac{1}{z_0}.
\]
Because
\(\displaystyle
z_0 = \frac{1}{2}\log\frac{1}{x} = -\frac{1}{2}\log x,
\)
we obtain
\(\displaystyle
\frac{1}{z_0}
= \frac{2}{\log(1/x)}
= \frac{2}{-\log x}.
\)
This proves the claim.
\end{proof}

\begin{lemma}\label[lemma]{lemma:B-ratio}
For every $x\in(0,1)$ we have
\begin{equation}\label{eq:B-ratio}
  \lim_{m\to\infty}\frac1m \frac{B_{m+1}(x)}{B_m(x)}
  = \frac{2(1-x)}{-\log x}.
\end{equation}
\end{lemma}

\begin{proof}

From \cite[Lemma 2.3, Eq.7]{TallaWaffo2026arxiv2602.16761}, we know
\[
  \sum_{k=0}^\infty (2k+1)^m x^k
  = \frac{1}{(1-x)^{m+1}} B_m(x), \qquad |x|<1,
\]
Putting it in place, we have $B_m(x)=(1-x)^{m+1}S_m(x)$. Hence
\[
  \frac{B_{m+1}(x)}{B_m(x)}
  = (1-x)\frac{S_{m+1}(x)}{S_m(x)}.
\]
Combining this with \cref{lemma:Sm-ratio} yields \eqref{eq:B-ratio}..
\end{proof}

\begin{lemma}\label[lemma]{lem:log-derivative-Xi}
Let $n\in\mathbb{N}$ and $0<x<1$. Consider \(\widetilde{\Xi}_n\) as defined in \eqref{eq:Xi-def}.

Set
\(\displaystyle
y(x):=-\,\frac{1-\sqrt{x}}{1+\sqrt{x}}.
\)
Then, for every $x$ with $B_{2n-1}(y(x))\neq0$, the logarithmic derivative of
$\widetilde{\Xi}_n$ is given by
\[
\frac{\widetilde{\Xi}_n'(x)}{\widetilde{\Xi}_n(x)}
=
\frac{2n-1}{2\sqrt{x}\,(1+\sqrt{x})}
-\frac{1}{2x}
-\frac{1}{4\sqrt{x}\,(1-\sqrt{x})}
\left[
  \frac{B_{2n}\!\bigl(y(x)\bigr)}{B_{2n-1}\!\bigl(y(x)\bigr)}
  -\bigl((4n-1)\,y(x)+1\bigr)
\right].
\]
\end{lemma}

\begin{proof}
Write
\[
\widetilde{\Xi}_n(x)
=C_n\,
\frac{(1+\sqrt{x})^{2n-1}}{\sqrt{x}}\,
B_{2n-1}\!\bigl(y(x)\bigr),
\qquad
C_n:=\frac{(-1)^{n+1}}{2^{4n-1}(2n-1)!},
\quad
y(x):=-\frac{1-\sqrt{x}}{1+\sqrt{x}}.
\]
Taking logarithms,
\[
\log \widetilde{\Xi}_n(x)
= \log C_n
+ (2n-1)\log(1+\sqrt{x})
-\tfrac12\log x
+ \log B_{2n-1}\!\bigl(y(x)\bigr).
\]
Differentiating with respect to $x$ gives
\[
\frac{\widetilde{\Xi}_n'(x)}{\widetilde{\Xi}_n(x)}
=
\frac{2n-1}{2\sqrt{x}\,(1+\sqrt{x})}
-\frac{1}{2x}
+\frac{B_{2n-1}'\!\bigl(y(x)\bigr)}{B_{2n-1}\!\bigl(y(x)\bigr)}\,y'(x).
\]
A direct computation yields
\[
y'(x)=\frac{1}{\sqrt{x}\,(1+\sqrt{x})^{2}},
\]
so
\begin{equation}\label{eq:Xi-log-basic}
\frac{\widetilde{\Xi}_n'(x)}{\widetilde{\Xi}_n(x)}
=
\frac{2n-1}{2\sqrt{x}\,(1+\sqrt{x})}
-\frac{1}{2x}
+\frac{1}{\sqrt{x}\,(1+\sqrt{x})^{2}}\,
\frac{B_{2n-1}'\!\bigl(y(x)\bigr)}{B_{2n-1}\!\bigl(y(x)\bigr)}.
\end{equation}

Next, we use the recurrence for the Eulerian polynomials of type~$B$\cite[Theorem 4.1]{ChenTangZhao2009} \cite[Theorem 3.4]{Brenti1994}:
\[
B_m(z)
=
\bigl((2m-1)z+1\bigr)B_{m-1}(z)
+2z(1-z)B_{m-1}'(z).
\]
Dividing by $B_{m-1}(z)$ and then replacing $m$ by $m+1$ gives
\[
\frac{B_m'(z)}{B_m(z)}
=
\frac{1}{2z(1-z)}
\left(
  \frac{B_{m+1}(z)}{B_m(z)}
  -\bigl((2m+1)z+1\bigr)
\right),
\]
for all $z$ with $B_m(z)\neq0$.  
Taking $m=2n-1$ and $z=y(x)$, we obtain
\[
\frac{B_{2n-1}'\!\bigl(y(x)\bigr)}{B_{2n-1}\!\bigl(y(x)\bigr)}
=
\frac{1}{2y(x)\bigl(1-y(x)\bigr)}
\left(
  \frac{B_{2n}\!\bigl(y(x)\bigr)}{B_{2n-1}\!\bigl(y(x)\bigr)}
  -\bigl((4n-1)y(x)+1\bigr)
\right).
\]
Furthermore, a short calculation shows
\[
y(x)\bigl(1-y(x)\bigr)
=
-\frac{2(1-\sqrt{x})}{(1+\sqrt{x})^{2}},
\]
so that
\[
\frac{1}{\sqrt{x}\,(1+\sqrt{x})^{2}}\cdot
\frac{1}{2y(x)\bigl(1-y(x)\bigr)}
=
-\frac{1}{4\sqrt{x}\,(1-\sqrt{x})}.
\]
Substituting this into \eqref{eq:Xi-log-basic} yields
\[
\frac{\widetilde{\Xi}_n'(x)}{\widetilde{\Xi}_n(x)}
=
\frac{2n-1}{2\sqrt{x}\,(1+\sqrt{x})}
-\frac{1}{2x}
-\frac{1}{4\sqrt{x}\,(1-\sqrt{x})}
\left[
  \frac{B_{2n}\!\bigl(y(x)\bigr)}{B_{2n-1}\!\bigl(y(x)\bigr)}
  -\bigl((4n-1)y(x)+1\bigr)
\right],
\]
which is the claimed formula.
\end{proof}

\vspace{0.3cm}

\section{Limit of the Stieltjes transform, limit density and limit distribution}\label[section]{sec:limits}

\vspace{0.3cm}

For the empirical zero measures $\mu_n$ in \eqref{eq:mu-n-def}, the normalized logarithmic derivative
\(\displaystyle
  s_n(z) := \frac{1}{n-1}\frac{\widetilde{\Xi}_n'(z)}{\widetilde{\Xi}_n(z)}
\)
is the Stieltjes transform of $\mu_n$, i.e.
\begin{equation}\label{eq:sn-Stieltjes}
  s_n(z) = \int_{\mathbb{R}} \frac{1}{z-x}\,\mu_n(dx), 
  \qquad z\in\mathbb{C}\setminus[0,1].
\end{equation}

\begin{lemma}\label[lemma]{lem:s-n-limit}
Let
\[
s_n(z):=\frac{1}{n-1}\,\frac{\widetilde{\Xi}_n'(z)}{\widetilde{\Xi}_n(z)},
\qquad z\in\mathbb{C}\setminus[0,1],
\]
where $\widetilde{\Xi}_n$ is defined in \eqref{eq:Xi-def}.
Fix an analytic branch of $\sqrt{z}$ on $\mathbb{C}\setminus[0,1]$ and set
\[
u(z):=\frac{\sqrt{z}-1}{\sqrt{z}+1}.
\]
Then $|u(z)|<1$ for $z\in\mathbb{C}\setminus[0,1]$, and for every such $z$ we have
\[
s_n(z)\longrightarrow s(z)
\quad\text{locally uniformly,}
\]
where
\[
s(z)
=
\frac{1}{\sqrt{z}\,\bigl(1+\sqrt{z}\bigr)}
+\frac{1}{\sqrt{z}\,\bigl(1-\sqrt{z}\bigr)}
\left(
  u(z)+\frac{1-u(z)}{\log u(z)}
\right).
\]
\end{lemma}

\begin{proof}
From \cref{lem:log-derivative-Xi} we have, for $z\in\mathbb{C}\setminus[0,1]$,
\[
\frac{\widetilde{\Xi}_n'(z)}{\widetilde{\Xi}_n(z)}
=
\frac{2n-1}{2\sqrt{z}\,(1+\sqrt{z})}
-\frac{1}{2z}
-\frac{1}{4\sqrt{z}\,(1-\sqrt{z})}
\left[
  \frac{B_{2n}(u(z))}{B_{2n-1}(u(z))}
  -\bigl((4n-1)u(z)+1\bigr)
\right],
\]
where $B_m$ is the Eulerian polynomial of type~$B$ and $u(z)=(\sqrt{z}-1)/(\sqrt{z}+1)$.

Dividing by $n-1$ we obtain
\[
\begin{aligned}
s_n(z)
&=
\frac{2n-1}{n-1}\,\frac{1}{2\sqrt{z}\,(1+\sqrt{z})}
-\frac{1}{2z(n-1)} \\[2mm]
&\hphantom{=} \;
-\frac{1}{4\sqrt{z}\,(1-\sqrt{z})}
\left[
  \frac{1}{n-1}\frac{B_{2n}(u(z))}{B_{2n-1}(u(z))}
  -\frac{4n-1}{n-1}u(z)
  -\frac{1}{n-1}
\right].
\end{aligned}
\]

By \cref{lemma:B-ratio}, for every $x$ with $x \in (0,1)$,
\[
\lim_{m\to\infty}\frac{1}{m}\frac{B_{m+1}(x)}{B_m(x)}
= \frac{2(1-x)}{-\log x},
\]
and this convergence is locally uniform in $x$.  
Since $|u(z)|<1$ on $\mathbb{C}\setminus[0,1]$ and $u$ is analytic there, we may apply this with $m=2n-1$ and $x=u(z)$ to get
\[
\lim_{n\to\infty}\frac{1}{2n-1}\frac{B_{2n}(u(z))}{B_{2n-1}(u(z))}
= \frac{2\bigl(1-u(z)\bigr)}{-\log u(z)},
\]
locally uniformly in $z$. Hence
\[
\lim_{n\to\infty}\frac{1}{n-1}\frac{B_{2n}(u(z))}{B_{2n-1}(u(z))}
= \lim_{n\to\infty}\frac{2n-1}{n-1}
   \cdot\frac{1}{2n-1}\frac{B_{2n}(u(z))}{B_{2n-1}(u(z))}
= 2\cdot\frac{2(1-u(z))}{-\log u(z)}
= \frac{4(1-u(z))}{-\log u(z)}.
\]

Moreover,
\[
\frac{2n-1}{n-1}\longrightarrow 2,\qquad
\frac{4n-1}{n-1}\longrightarrow 4,\qquad
\frac{1}{n-1}\longrightarrow 0,
\]
all locally uniformly in $z$. Inserting these limits into the expression for $s_n(z)$ we obtain
\[
\begin{aligned}
s(z)
&= \lim_{n\to\infty} s_n(z) \\[1mm]
&= \frac{1}{\sqrt{z}\,(1+\sqrt{z})}
-\frac{1}{4\sqrt{z}\,(1-\sqrt{z})}
\left[
  \frac{4(1-u(z))}{-\log u(z)}
  -4u(z)
\right].
\end{aligned}
\]

Using $(1-u)/(-\log u)=-(1-u)/\log u$, this simplifies to
\[
s(z)
= \frac{1}{\sqrt{z}\,(1+\sqrt{z})}
+\frac{1}{\sqrt{z}\,(1-\sqrt{z})}
\left(
  u(z)+\frac{1-u(z)}{\log u(z)}
\right),
\]
which is the claimed formula.  
All limits above are locally uniform, and sums and products of locally uniformly convergent sequences preserve local uniform convergence, hence $s_n\to s$ locally uniformly on $\mathbb{C}\setminus[0,1]$.
\end{proof}

\begin{theorem}[Limiting density]\label{thm:limiting-density}
Let $\widetilde{\Xi}_n$ be defined by \eqref{eq:Xi-def}, and let
$\mu_n$ be the empirical zero measure from \eqref{eq:mu-n-def}.
Then $\mu_n$ converges weakly, as $n\to\infty$, to a probability measure
$\mu$ supported on $(0,1)$ with density
\begin{equation}\label{eq:rho-density}
  \rho(x)
  =
  \frac{2}{
    \sqrt{x}\,(1-x)
    \bigl(\log^2\!\frac{1-\sqrt{x}}{1+\sqrt{x}} + \pi^2\bigr)
  },
  \qquad 0<x<1,
\end{equation}
and $\rho(x)=0$ for $x\notin(0,1)$.
\end{theorem}

\begin{proof}
By definition, the normalized logarithmic derivative
\[
  s_n(z) := \frac{1}{n-1}\,\frac{\widetilde{\Xi}_n'(z)}{\widetilde{\Xi}_n(z)},
  \qquad z\in\mathbb{C}\setminus[0,1],
\]
is the Stieltjes transform of $\mu_n$, i.e.
\[
  s_n(z) = \int_{\mathbb{R}} \frac{1}{z-x}\,\mu_n(dx),
  \qquad z\in\mathbb{C}\setminus[0,1].
\]

\Cref{lem:s-n-limit} shows that $s_n(z)\to s(z)$ locally uniformly on
$\mathbb{C}\setminus[0,1]$, where
\begin{equation}\label{eq:s-limit}
  s(z)
  =
  \frac{1}{2\sqrt{z}\,(1+\sqrt{z})}
  +
  \frac{1}{\sqrt{z}\,(1-\sqrt{z})}
  \left(
    u(z) + \frac{1-u(z)}{\log u(z)}
  \right),
\end{equation}
and
\(\displaystyle
  u(z):=\frac{\sqrt{z}-1}{\sqrt{z}+1},
\)
for a fixed analytic branch of $\sqrt{z}$ on $\mathbb{C}\setminus[0,1]$.
In particular, $s$ is analytic on $\mathbb{C}\setminus[0,1]$.

Each $\mu_n$ is a probability measure, so
$s_n(z) = \frac{1}{z} + O(z^{-2})$ as $z\to\infty$, and the convergence
$s_n\to s$ implies
\(\displaystyle
  s(z) = \frac{1}{z} + O(z^{-2}), \qquad z\to\infty.
\)
By standard properties of Stieltjes transforms (see, e.g.,
\cite[Chap.~III]{Billingsley-convergence}), there exists a unique
probability measure $\mu$ supported on $[0,1]$ whose Stieltjes
transform is $s$, and $\mu_n\Rightarrow\mu$ weakly.

It remains to identify the density of $\mu$ on $(0,1)$.  
For $x\in(0,1)$ and $\varepsilon>0$ we consider $s(x+i\varepsilon)$
with $\varepsilon$ tending to $0$ from above.  
The Stieltjes inversion formula yields
\[
  \rho(x)
  = -\frac{1}{\pi}\,\lim_{\varepsilon\rightarrow 0^+}\Im s(x+i\varepsilon),
  \qquad 0<x<1.
\]

For $x\in(0,1)$ we have $\sqrt{x}>0$ and
\[
  u(x) = \frac{\sqrt{x}-1}{\sqrt{x}+1} \in (-1,0).
\]
Writing
\[
  q(x):=\frac{1-\sqrt{x}}{1+\sqrt{x}}\in(0,1),
  \qquad u(x) = -q(x),
\]
and using the branch of the logarithm continuous from the upper
half-plane, we obtain
\[
  \lim_{\varepsilon\rightarrow 0^+}\log u(x+i\varepsilon)
  =
  \log q(x) + i\pi,
\]
where $\log$ on the right-hand side denotes the real natural logarithm
of $q(x)\in(0,1)$.  Hence
\[
  \lim_{\varepsilon\rightarrow 0^+}\frac{1}{\log u(x+i\varepsilon)}
  =
  \frac{\log q(x) - i\pi}{\log^2 q(x) + \pi^2}.
\]

Substituting this limit into \eqref{eq:s-limit} with $z=x+i\varepsilon$,
and noting that all remaining terms in \eqref{eq:s-limit} converge to
real limits as $\varepsilon\rightarrow 0^+$, a straightforward computation
of the limit of the imaginary parts gives
\[
  -\frac{1}{\pi}\,\lim_{\varepsilon\rightarrow 0^+}\Im s(x+i\varepsilon)
  =
  \frac{2}{
    \sqrt{x}\,(1-x)
    \bigl(\log^2\!\tfrac{1-\sqrt{x}}{1+\sqrt{x}} + \pi^2\bigr)
  }.
\]
Thus $\mu$ has density $\rho(x)$ on $(0,1)$ given by \eqref{eq:rho-density}
and vanishes off $(0,1)$, which completes the proof.
\end{proof}

\begin{theorem}[Limiting zero distribution]\label{thm:limiting-zero-distribution}
Let $\widetilde{\Xi}_n$ be defined by \eqref{eq:Xi-def}, and let
$\mu_n$ be the empirical zero measure from \eqref{eq:mu-n-def}.
Let $F_n$ denote the distribution function of $\mu_n$, i.e.
\[
F_n(x):=\mu_n((-\infty,x]),\qquad x\in\mathbb{R}
\]
Then $F_n(x)$ converges pointwise on $\mathbb{R}$ to the function
\[
F(x):=
\begin{cases}
0, & x\le 0,\\[1mm]
\displaystyle \frac{2}{\pi}\arctan\!\left(
  \frac{1}{\pi}\log\frac{1+\sqrt{x}}{1-\sqrt{x}}
\right), & 0<x<1,\\[3mm]
1, & x\ge 1.
\end{cases}
\]
In particular, the empirical distribution functions $F_n$ associated
with the real zeros of $\widetilde{\Xi}_n$ in $(0,1)$ converge to $F$
at every continuity point of $F$.
\end{theorem}

\begin{proof}
By \cref{thm:limiting-density}, the measures $\mu_n$ converge
weakly to a probability measure $\mu$ supported on $(0,1)$ with density
\[
\rho(x)=\frac{2}{
  \sqrt{x}\,(1-x)
  \bigl(\log^2\!\tfrac{1-\sqrt{x}}{1+\sqrt{x}}+\pi^2\bigr)
},
\qquad 0<x<1,
\]
and $\rho(x)=0$ for $x\notin(0,1)$.  Hence the distribution function of
$\mu$ is
\[
F(x)=\mu((-\infty,x])=
\begin{cases}
0, & x\le 0,\\[1mm]
\displaystyle \int_0^x \rho(t)\,dt, & 0<x<1,\\[3mm]
1, & x\ge 1.
\end{cases}
\]

For $0<x<1$ we define
\(\displaystyle
G(x):=\frac{2}{\pi}\arctan\!\left(
  \frac{1}{\pi}\log\frac{1+\sqrt{x}}{1-\sqrt{x}}
\right).
\)
Note that $G(0+)=0$.  It suffices to show that $G'(x)=\rho(x)$ on
$(0,1)$, because then $G(x)=\int_0^x\rho(t)\,dt$ and hence $F(x)=G(x)$
for $0<x<1$. Consequently,
\[
F(x)=\int_0^x \rho(t)\,dt = G(x)
\quad\text{for }0<x<1,
\]
while $F(x)=0$ for $x\le0$ and $F(x)=1$ for $x\ge1$.

Since $\mu_n\Rightarrow\mu$ and $F$ is continuous on $\mathbb{R}$, standard
results on weak convergence of probability measures (e.g.\
\cite[Thm.~2.3]{Billingsley-convergence}) imply that $F_n(x)\to F(x)$ at
every continuity point of $F$. This completes the proof.
\end{proof}

\vspace{0.4cm}

\section{Interpretation and related conjecture}\label[section]{sec:interpretation}

\vspace{0.4cm}

\textit{In \cite{TallaWaffo2026arxiv2602.16761} we conjectured that the smallest zero of $\widetilde{\Xi}_n$ converges to $0$. 
However, the theorem proved there does not fully capture the specific properties and structure of this polynomial family. 
In this short section we confirm the conjecture.}

\vspace{0.2cm}

Using the limiting density from \cref{thm:limiting-density} and the explicit distribution function in \cref{thm:limiting-zero-distribution}, we obtain a fairly detailed picture of the zero distribution of the polynomials $\widetilde{\Xi}_n$. The weak convergence $\mu_n\Rightarrow\mu$ with density $\rho$ supported on $(0,1)$ implies that, asymptotically, almost all real zeros lie in the open unit interval and their empirical distribution is governed by~$\rho$. In particular, the empirical distribution functions $F_n$ converge pointwise to the limiting distribution function $F$, so the $k$-th ordered zero $x_{k,n}\in(0,1)$ satisfies the asymptotic relation
\[
  \frac{k}{n}\approx F(x_{k,n}),\qquad n\to\infty,
\]
at every continuity point of $F$. Thus the limiting law $F$ acts as a deterministic profile for the ordered zeros.

\vspace{0.2cm}

The shape of $\rho$ reveals a pronounced clustering of zeros near both endpoints $0$ and $1$. Near the origin we have
\[
  \rho(x)\sim \frac{2}{\pi^2\sqrt{x}},\qquad x\rightarrow 0^+,
\]
so the density diverges like $x^{-1/2}$, but in an integrable way. Integrating this asymptotic behaviour and inverting the relation $F(x)\approx \tfrac{4}{\pi^2}\sqrt{x}$ as $x\downarrow0$ shows that, for fixed $k$ and $n\to\infty$,
\[
  x_{k,n}\sim \frac{\pi^4}{16}\,\frac{k^2}{n^2}.
\]
In particular, the smallest zero has size of order $n^{-2}$; more generally, the leftmost zeros form a quadratic edge with approximately parabolic spacing when viewed on the scale $k/n$.

\vspace{0.2cm}

The behaviour near $x=1$ is more singular. Writing
\[
  L(x):=\log\frac{1+\sqrt{x}}{1-\sqrt{x}},
\]
we have $L(x)\to+\infty$ as $x\uparrow1$, and the distribution function
\[
  F(x)=\frac{2}{\pi}\arctan\!\Bigl(\frac{1}{\pi}L(x)\Bigr)
\]
satisfies
\[
  1-F(x)\sim \frac{2}{L(x)},\qquad x\uparrow 1.
\]
Since $L(x)$ grows like $|\log(1-\sqrt{x})|$, this shows that the right tail of the limiting distribution decays only logarithmically in the distance to~$1$. Translating this into information on the ordered zeros, we see heuristically that if $n-k$ is fixed while $n\to\infty$, then
\[
  1-F(x_{k,n})\approx \frac{n-k}{n}\sim \frac{2}{L(x_{k,n})},
\]
so $L(x_{k,n})$ is of order $n$, and consequently $1-x_{k,n}$ is exponentially small in~$n$. Thus the rightmost zeros approach $1$ at an essentially exponential rate, much faster than any power of $n^{-1}$. This strong pile-up near $x=1$ is the analytic manifestation of the logarithmic singularity in the density.

\vspace{0.2cm}

In summary, the zeros of $\widetilde{\Xi}_n$ accumulate on the interval $(0,1)$ with a density that blows up at both endpoints, but in two different regimes: a square-root edge at $0$ and a logarithmically corrected edge at $1$. The smallest zeros live on the natural $n^{-2}$ scale, whereas the largest zeros are squeezed exponentially close to $1$. In between these two extremes, the spacing of zeros is of order $1/n$ and is well described by the smooth limiting density $\rho(x)$.

\vspace{0.4cm}

\section{Numerical experiments}\label[section]{sec:numerics}

We close this work with a numerical illustration of the cumulative distribution functions associated with the real zeros of $\widetilde{\Xi}_n$ in $(0,1)$.

Rather than emphasizing high-precision computations of individual zeros, we focus on the empirical distribution functions $F_n$ and their visual comparison with the limiting distribution $F$. The plots clearly indicate the convergence predicted by \cref{thm:limiting-density} and \cref{thm:limiting-zero-distribution}: as $n$ increases, the empirical CDFs become progressively closer to the limiting curve.

In particular, one observes that the smallest zeros drift toward $0$ on the natural $n^{-2}$–scale, while the largest zeros cluster rapidly near $1$, reflected in the steep rise of the CDF close to the boundary.

\begin{center}
	\includegraphics[width=0.95\textwidth]{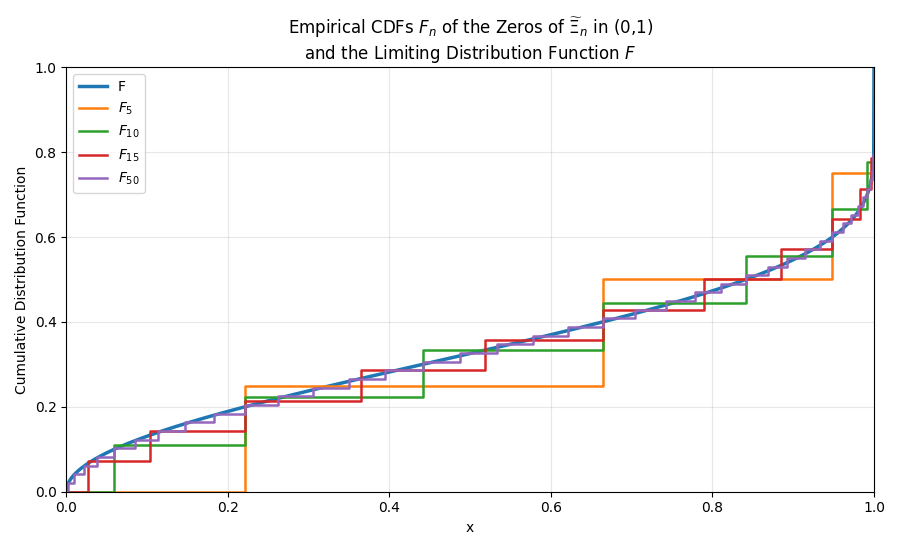}
\end{center}

\section*{Acknowledgments}
The author acknowledges the use of an AI language model for assistance with literature search, presentation of the manuscript, verification of results and clarifying standard probability notions.

\printbibliography

\end{document}